\documentclass{amsart}

\RequirePackage{fix-cm}
\usepackage{graphicx}
\usepackage{amsmath, amsopn,amstext,amscd,amsfonts,amssymb}
\usepackage{dsfont}
\usepackage{comment}
\usepackage[active]{srcltx}
\usepackage{graphicx, epsfig, subfig}

\usepackage{textcomp}

\usepackage{algorithm}

\makeatletter
\def\BState{\State\hskip-\ALG@thistlm}
\makeatother

\def\downbar#1{
\setbox10=\hbox{$#1$}
            \dimen10=\ht10 \advance\dimen10 by 2.5pt
            \ifdim \dimen10<15pt 
               \advance\dimen10 by -0.5pt
               \dimen11=\dimen10
               \advance\dimen10 by 2.5pt
               \lower \dimen11
            \else \lower \ht10 \fi
            \hbox {\hskip 1.5pt \vrule height \dimen10 depth \dp10}}
\def\upbar#1{
\setbox10=\hbox{$#1$}
            \dimen10=\ht10 \advance\dimen10 by \dp10 \advance\dimen10 by 2.5pt
            \ifdim \dimen10<15pt 
                \advance\dimen10 by 2pt \fi
            \raise 2.5pt \hbox {\hskip -1.5pt \vrule height \dimen10}}

\usepackage{multicol}
\usepackage{colortbl}
\usepackage{exscale}
\usepackage{amssymb,latexsym,amsthm,amsfonts,color,fancyhdr,mathrsfs}
\usepackage{lscape}
\usepackage{rotating}
\usepackage{caption}

\newtheorem{definition}{\bf Definition}[section]
\newtheorem{theorem}{\bf Theorem}[section]
\newtheorem{proposition}{\bf Proposition}[section]
\newtheorem{lemma}{\bf Lemma}[section]

\newtheorem{corollary}{\bf Corollary}[section]
\newtheorem{remark}{\bf Remark}[section]

\numberwithin{equation}{section}
\begin{document}
\title[On classical orthogonal polynomials on bi-lattices]{On classical orthogonal polynomials on bi-lattices}

\author{K. Castillo}
\address{University of Coimbra, CMUC, Dep. Mathematics, 3001-501 Coimbra, Portugal}
\email{kenier@mat.uc.pt}
\author{G. Filipuk}
\address{Institute of Mathematics, University of Warsaw, ul.  Banacha 2,02-097, Warsaw, Poland}
\email{filipuk@mimuw.edu.pl }
\author{D. Mbouna}
\address{Institute of Mathematics, University of Warsaw, ul.  Banacha 2,02-097, Warsaw, Poland}
\email{dmbouna@mimuw.edu.pl}

\date{\today}

\begin{abstract}
In [J. Phys. A: Math. Theor. 45 (2012)], while  looking for spin chains that admit perfect state transfer, Vinet and Zhedanov found an apparently new sequence of orthogonal polynomials, 
that they called para-Krawtchouk polynomials, defined on a bilinear lattice. In this note we present necessary and sufficient conditions for the regularity of solutions of the corresponding functional equation. Moreover, the functional Rodrigues formula and a closed formula for the recurrence coefficients are presented.  As a consequence, we characterize all solutions of the functional equation, including as very particular cases the Meixner, Charlier, Krawtchouk, Hahn, and para-Krawtchouk polynomials.
\end{abstract}
\maketitle

\section{Definitions and mains results}\label{introduction}
Let $\mathcal{P}$ be the vector space of all polynomials with complex coefficients
and let $\mathcal{P}^*$ be its algebraic dual.
$\mathcal{P}_n$ denotes the space of all polynomials with degree less than or equal to $n$.
Define $\mathcal{P}_{-1}:=\{0\}$.
A simple set in $\mathcal{P}$ is a sequence $(P_n)_{n\geq0}$ such that
$P_n\in\mathcal{P}_n\setminus\mathcal{P}_{n-1}$ for each $n$.
A simple set $(P_n)_{n\geq0}$ is called a sequence of orthogonal polynomials (OP)
with respect to ${\bf u}\in\mathcal{P}^*$ if 
$$
\langle{\bf u},P_nP_k\rangle=h_n\delta_{n,k}\quad(n,k=0,1,\ldots;\;h_n\in\mathbb{C}\setminus\{0\}),
$$
where $\langle{\bf u},f\rangle$ is the action of ${\bf u}$ on $f\in\mathcal{P}$. A linear form $\bf u$  is called regular if there exists an OP with respect to it. One of the major questions in the theory of orthogonal polynomials is to identify properties which characterize specific OP.  In this sense the best-known examples are the characterization theorems for the old classical orthogonal polynomials (Hermite, Laguerre, Jacobi, and Bessel); particularly those due to Al-Salam and Chihara, Bochner, Hahn, Maroni, and McCarthy. The simplest way to prove these theorems is by using the algebraic theory developed by Maroni (see \cite{M1991}). Recall that motivated by Hahn's characterization, Geronimus proved the following theorem \cite[Theorem II]{G1940}: 

\vspace{2mm}
\begin{changemargin}{0.8cm}{0.8cm} 
{\em A sequence of polynomials and the sequence of its derivatives are orthogonal with respect to sequences of moments\footnote{
Recall that given a sequence of complex numbers $(u_n)_{n\geq0}$, $(P_n)_{n\geq0}$ is said to be orthogonal with respect to $(u_n)_{n\geq0}$ if it is an OPS with respect to the functional ${\bf u}\in\mathcal{P}^*$ defined by 
$$
\langle {\bf u},f\rangle:=\sum_{k=0}^na_{k}u_k,\quad f(z)=\sum_{k=0}^na_{k}z^k\in\mathcal{P}.
$$}
 $(u_n)_{n\geq0}$ and $(v_n)_{n\geq0}$, respectively, if and only if 
\begin{equation}\label{A1}
(na+d)u_{n+1}+(nb+e)u_n+ncu_{n-1}=0 
\end{equation}
for each $n$, where $a,b,c,d,e$ are complex numbers such that
$na+d\neq0$ and $\det\Big(u_{i+j}\Big)_{i,j=0}^n\neq0$, and} 
\begin{equation}\label{A3}
v_n=au_{n+2}+bu_{n+1}+cu_{n}.
\end{equation}
\end{changemargin}
\vspace{2mm}
Maroni, as Rota in his modern umbral calculus, observed that Geronimus' theorem is clearer if we write the moments of OP, $(u_n)_{n\geq 0}$, as defined by linear function functional ${\bf u}$ in $\mathcal{P}$: $u_n=\langle{\bf u},x^n\rangle$. The description of OP is then condense into the properties of the linear functional and, in particular, this allows us to rewrite the difference equation \eqref{A1} as a functional equation 
\begin{equation}\label{GPa}
{\bf D}(\phi{\bf u})=\psi{\bf u},
\end{equation}
where $\phi(z)=az^2+bz+c$ and $\psi(z)=dz+e$, and 
\eqref{A3} as ${\bf v}=\phi{\bf u}$, being ${\bf u}$ and ${\bf v}$ the linear functionals on $\mathcal{P}$ given by 
$\langle{\bf u},z^n\rangle=u_n$ and $\langle{\bf v},z^n\rangle=v_n$. Each regular solution of \eqref{GPa} is called a classical linear form. When we start from \eqref{GPa} there are, in terms of $\phi$ and $\psi$, two fundamental questions to answer: 
\vspace{2mm}
\begin{changemargin}{0.8cm}{0.8cm} 
{\em 
(1) Is ${\bf u}$ regular? (2) If yes, what are OP families associated with ${\bf u}$? }\end{changemargin}
\vspace{2mm}
In the case of \eqref{GPa} these issues are solved in relatively easily way, see \cite{CP23} for a recent survey on the subject.

Of course, \eqref{GPa} does not include families as important as the Askey-Wilson or Racah polynomials and this led in \cite{KCDMJP2022} to define and study classical OP from the corresponding distributional equation on lattices $x(s)$ given by
\begin{equation}
\label{xs-def}
x(s):=\left\{
\begin{array}{ccl}
\mathfrak{c}_1 q^{-s} +\mathfrak{c}_2 q^s +\mathfrak{c}_3 & {\rm if} &  q\neq1,\\ [7pt]
\mathfrak{c}_4 s^2 + \mathfrak{c}_5 s +\mathfrak{c}_6 & {\rm if} &  q =1
\end{array}
\right.
\end{equation}
($s\in\mathbb{C}$), where $q>0$ (fixed) and $\mathfrak{c}_j$ ($1\leq j\leq6$) are (complex) constants,
that may depend on $q$, such that $(\mathfrak{c}_1,\mathfrak{c}_2)\neq(0,0)$ if $q\neq1$,
and $(\mathfrak{c}_4,\mathfrak{c}_5,\mathfrak{c}_6)\neq(0,0,0)$ if $q=1$. For $x(s)=\mathfrak{c}_6$ we recover \eqref{GPa}. In addition,  \cite{LVAZ2012}, while  looking for spin chains that admit perfect state transfer, Vinet and Zhedanov found an apparently new OP, 
the so-called para-Krawtchouk polynomials, defined on the bi-lattice
\begin{align}\label{original}
x(s)=s+\frac12(\mu-1) (1+(-1)^s)\quad (s=0,1,\dots, N),
\end{align}
where $N$ is odd and the $\mu \in (0,2)$. Examples of semi-classical OP on bi-lattices can be find in  \cite{CSWV2012}. However, Vinet and Zhedanov pointed out that the para-Krawtchouk polynomials are ``classical" OP. Before giving a rigorous definition in the distributional sense of this last statement, let us introduce some necessary definitions.
\begin{definition}\label{def-DxSxNUL} 
We call the complex valued function $x$ given by
\begin{align}
x(s)=s+\gamma(1+(-1)^s)\quad (s, \gamma \in \mathbb{C})
\end{align}
a bi-lattice.
\end{definition}
\begin{definition}
Define two operator on $\mathcal{P}$, $\mathrm{D}_s$ and $\mathrm{S}_s$, as follows:
\begin{align*}
\mathrm{D}_s f(x(s))&= \frac{f(x(s+1))-f(x(s-1))}{2}, \\[7pt]
\mathrm{S}_s f(x(s))&= \frac{f(x(s+1)) + f(x(s-1))}{2}.
\end{align*}
\end{definition}

\begin{definition} 
For each ${\bf u}\in\mathcal{P}^*$, the functionals $\mathbf{D}_s{\bf u}\in\mathcal{P}^*$
and $\mathbf{S}_s{\bf u}\in\mathcal{P}^*$ are defined by
\begin{align*}
\langle \mathbf{D}_s{\bf u},f\rangle:=-\langle {\bf u},\mathrm{D}_s f\rangle,\quad
\langle \mathbf{S}_s{\bf u},f\rangle:=\langle {\bf u},\mathrm{S}_s f\rangle\quad (f\in\mathcal{P}).
\end{align*}
\end{definition}

Let us now define what we understand here by classical orthogonal polynomials on a bi-lattice.

\begin{definition}\label{NUL-def}
${\bf u}\in\mathcal{P}^*$ is called $bi$-classical if it is regular and there exist nonzero polynomials
$\phi\in\mathcal{P}_2$ and $\psi\in\mathcal{P}_1$
such that
\begin{equation}\label{NUL-Pearson}
\mathbf{D}_s(\phi{\bf u})=\mathbf{S}_s(\psi{\bf u}).
\end{equation}
An OP with respect to {\bf u} will be called a classical OP on a bi-lattice.
\end{definition}

The following theorem answers in this context to question (1) posed above.

\begin{theorem}\label{main-Thm1}
Consider the bi-lattice
$$
x(s)=s+\gamma (1+(-1)^s) \quad (s, \gamma \in \mathbb{C}). 
$$
Let ${\bf u}\in\mathcal{P}^*\setminus\{{\bf 0}\}$ and suppose that there exist $(\phi,\psi)\in\mathcal{P}_2\times\mathcal{P}_1\setminus\{(0,0)\}$ 
such that 
\begin{equation}\label{NUL-PearsonMainThm1}
\mathbf{D}_{s}(\phi{\bf u})=\mathbf{S}_{s}(\psi{\bf u}).
\end{equation}
Set $\phi(z)=az^2+bz+c$ and $\psi(z)=dz+e$ ($a,b,c,d,e \in \mathbb{C}$). The functional ${\bf u}$ is regular  if and only if the following conditions hold:
\begin{equation}\label{le1a}
d_n\neq0,\quad \phi \left(-\frac{e_n}{d_{2n}}\right)+n d_n\neq0,
\end{equation}
where $d_n=an+d$, $e_n=bn+e$.
Under such conditions, the monic OPS $(P_n)_{n\geq 0}$ with respect to ${\bf u}$ satisfies 
\begin{equation}\label{ttrr-Dx}
P_{n+1}(z)=(z-B_n)P_n(z)-C_nP_{n-1}(z),
\end{equation}
with $P_{-1}(z)=0$, where 
\begin{align}
B_n & =\frac{n e_{n-1}}{d_{2n-2}}-\frac{(n+1)e_n}{d_{2n}},\label{Bn-Dx} \\[7pt]
C_{n+1} & =-\frac{(n+1)d_{n-1}}{d_{2n-1}d_{2n+1}}\left(\phi\left( -\frac{e_{n}}{d_{2n}}\right)+nd_n\right).\label{Cn-Dx}
\end{align}
Moreover, the following functional Rodrigues formula holds:
\begin{align}\label{RodThemMain}
P_n {\bf u} = k_n\mathbf{D}_s^n {\bf u}^{[n]}, \quad
k_n := (-1)^{n} \prod_{j=1} ^n d_{n+j-2} ^{-1}.
\end{align}
\end{theorem}

Let $\big(H_n ^{\gamma}(\cdot;a,b)\big)_{n\geq 0}$ be a sequence of OP satisfying \eqref{ttrr-Dx} with
\begin{align}\label{def-newOPS1}
B_n=-2n\sqrt{a+1},\quad C_{n+1}=(an+b)(n+1),
\end{align}
with $a,b \in \mathbb{C}$ such that $an+b\neq 0$. Let $\big(Q_n ^{\gamma}(\cdot;a,b,c)\big)_{n\geq 0}$ be a sequence of OP satisfying \eqref{ttrr-Dx} with
\begin{align}\label{def-newOPS2}
B_n&= \frac{(a-1)bc}{(n+a)(n+a-1)},\\[7pt]
C_{n+1}&=-\frac{(n+1)(n+a+b)(n+a-b)(n+a+c)(n+a-c)(n+2a-1)}{(2n+2a-1)(n+a)^2(2n+2a+1)},
\end{align}
with $a,b,c \in \mathbb{C}$ such that
$$-2a,-a-b,-a+b,-a-c,-a+c \notin \mathbb{N}_0.$$

The following theorem answers in this context to questions (2) posed above.

\begin{theorem}\label{main-Thm2}
Up to an affine transformation of the variable, the only classical OP on the bi-lattice \eqref{xs-def} are $\big(Q_n ^{\gamma}(\cdot;a,b,c)\big)_{n\geq 0}$ and $\big(H_n ^{\gamma}(\cdot;a,b)\big)_{n\geq 0}$.
\end{theorem}

It is worth pointing out that  the classical families for the linear lattice, Meixner, Charlier, Krawtchouk, and Hahn polynomials, are included in Theorem \ref{main-Thm2}, as well as the para-Krawtchouk polynomials.

\subsection*{Meixner polynomials}
The monic Meixner polynomials, $(M_n(\cdot;\beta,c))_{n\geq 0}$, satisfies (see \cite[p.234]{KLS2010}) \eqref{ttrr-Dx} with 
\begin{align*}
B_n =\frac{n+(n+\beta)c}{1-c},\quad C_{n+1}=\frac{(n+1)(n+\beta)}{(1-c)^2},
\end{align*}
{with $\beta >0$ and $0<c<1$. }It is easy to check that
{
\begin{align*}
&2^n M_n\left(\frac{1}{2}\sqrt{\frac{c+3}{c-1}}z+ \frac{\beta c}{1-c},\beta,c\right)\\[7pt]
&=\Big( \frac{c+3}{c-1} \Big)^{n/2} H_n ^{\gamma} \left(z;\frac{4}{(c+3)(c-1)},\frac{4\beta}{(c+3)(c-1)}  \right).
\end{align*}
}
\subsection*{Charlier polynomials}
The monic Charlier polynomials, $(C_n(\cdot;a))_{n\geq 0}$, satisfies (see \cite[p.247]{KLS2010}) \eqref{ttrr-Dx} with 
\begin{align*}
B_n =n+a,\quad C_{n+1}=a(n+1),
\end{align*}
{with $a>0$.} It is easy to check that
{
\begin{align*}
C_n(z,a)=(-2)^{-n} H_n ^{\gamma} \left(-2(z-a);0,4a \right).
\end{align*}
}

\subsection*{Krawtchouk polynomials}
The monic Krawtchouk polynomials, $(K_n(\cdot;p,N))_{n\geq 0}$, satisfies (see \cite[p.237]{KLS2010}) \eqref{ttrr-Dx} with 
\begin{align*}
B_n =n(1-p)+p(N-n),\quad C_{n+1}=(p-1)(n+1)(n-N),
\end{align*}
{with $N\in \mathbb{N}$ and $0<p<1$.} It is easy to check that
{
\begin{align*}
&K_n(z,p,N)\\
&=\frac{(4(p-1)^2+1)^{n/2}}{(-2)^n} H_n ^{\gamma} \left(\frac{-2}{\sqrt{4(p-1)^2+1}}\left(z-pN\right);\frac{4(p-1)}{4(p-1)^2+1},\frac{4(1-p)N}{4(p-1)^2+1}  \right).
\end{align*}
}
\subsection*{Hahn polynomials}
The monic Hahn polynomials, $(L_n(\cdot;\alpha,\beta,N))_{n\geq 0}$, satisfies (see \cite[p.204]{KLS2010}) \eqref{ttrr-Dx} with
\begin{align*}
B_n&=\frac{(n+\alpha +1)(N-n)(n+\alpha +\beta +1)}{(2n+\alpha +\beta +1)(2n+\alpha+\beta +2)} +\frac{n(n+\beta)(n+\alpha +\beta+N +1)}{(2n+\alpha +\beta )(2n+\alpha+\beta +1)},\\[7pt]
C_{n+1}&=\frac{(n+1){ (N-n)}(n+\alpha +\beta+1)(n+\alpha+1)(n+\beta+1)(n+\alpha+\beta+N+2)}{(2n+\alpha +\beta +1)(2n+\alpha +\beta +2)^2(2n+\alpha +\beta { +3})},
\end{align*}
{ with $-1<\alpha<-N$, $-1<\beta<-N$ and $N\in \mathbb{N}$}. It is easy to check that
{
\begin{align*}
&L_n(z,\alpha,\beta,N)\\[7pt]
&\quad =2^{-n} Q_n ^{\gamma} \left(2z+\frac{1}{2}(\alpha -\beta)-N;\frac{1}{2}(\alpha +\beta +2),\frac{1}{2}(\alpha -\beta),\frac{1}{2}(\alpha+\beta +2N+2)  \right).
\end{align*}
}
\subsection*{Para-Krawtchouk polynomials}
The monic para-Krawtchouk polynomials, $(P_n(\cdot;\mu,N))_{n\geq 0}$, satisfies (see \cite[(3.11)]{LVAZ2012}) \eqref{ttrr-Dx} with 
\begin{align*}
B_n={ \frac{N+\mu -1}{2}}, \quad C_{n+1}=-\frac{(n+1)(n-N)(2n+1-N-\mu)(2n+1-N+\mu)}{4(2n-N)(2n-N+2)},
\end{align*}
{ with $0<\mu <2$ and $N\in \mathbb{N}$}. It is easy to check that
\begin{align*}
P_n(z;\mu,N)={ Q_n^{\gamma}\left(z-\frac{N+\mu -1}{2};(1-N)/2, \mu /2,0  \right)}.
\end{align*}

In Section \ref{pre} we present some preliminary results. In Section \ref{th1} and Section \ref{th2} we prove Theorem \ref{main-Thm1} and Theorem \ref{main-Thm2}, respectively.

\section{Preliminary results}\label{pre}
We summarize some results obtained in \cite{KCDMJP2022}. When something new is added, we present a proof of the result.

\begin{lemma}\label{propertyDxSx}
Let $f,g\in\mathcal{P}$. Then the following properties hold:\rm
\begin{align}
 &\mathrm{D}_s \big(fg\big)= \big(\mathrm{D}_s f\big)\big(\mathrm{S}_s g\big)+\big(\mathrm{S}_s f\big)\big(\mathrm{D}_s g\big), \label{def-Dxfg} \\[7pt]
 &\mathrm{S}_s \big( fg\big)= \big(\mathrm{D}_s f\big) \big(\mathrm{D}_s g\big) +\big(\mathrm{S}_s f\big) \big(\mathrm{S}_s g\big), \label{def-Sxfg} \\[7pt]
&\mathrm{S}_s\mathrm{D}_s f= \mathrm{D}_s\mathrm{S}_s f, \label{def-DxSxf} \\[7pt]
&\mathrm{S}_s^2 f =  \mathrm{D}_s^2 f +f, \label{def-Sx2f} \\[7pt]
&f\mathrm{S}_sg =\mathrm{S}_s\big(g\mathrm{S}_sf \big)-\mathrm{D}_s\big(g\mathrm{D}_sf\big), \label{def-fSxg} \\[7pt]
&f\mathrm{D}_s g =\mathrm{D}_s\big(g\mathrm{S}_sf\big)-\mathrm{S}_s\big(g\mathrm{D}_sf\big), \label{def-fDxg}
\end{align}
\end{lemma}

\begin{lemma}
Let $f,g \in \mathcal{P}$ and ${\bf u}\in \mathcal{P}^*$. Then the following properties hold.
\begin{align}
&\mathbf{D}_s(f{\bf u}) =\mathrm{S}_s f \mathbf{D}_s {\bf u} +\mathrm{D}_s f \mathbf{S}_s {\bf u}, \label{def-Dxfu} \\[7pt]
&\mathbf{S}_s (f {\bf u})
=\mathrm{S}_s f \mathbf{S}_s{\bf u}+\mathrm{D}_sf {\bf D}_s {\bf u}, \label{def-Sxfu} \\[7pt]
&f\mathbf{D}_s  {\bf u} 
={\bf D}_s\big( \mathrm{S}_sf {\bf u} \big) -{\bf S}_s\big( \mathrm{D}_s f {\bf u}\big),  \label{def-fDxu} \\[7pt]
&f\mathbf{S}_s  {\bf u} 
={\bf S}_s\big( \mathrm{S}_sf {\bf u} \big) -{\bf D}_s\big( \mathrm{D}_s f {\bf u}\big),  \label{def-fSxu} \\[7pt]
& {\bf D}_s ^2{\bf u}={\bf S}_s ^2 {\bf u}-{\bf u}.\label{def-Dx2u-Sx2u}
\end{align}
\end{lemma}

\begin{proposition}
Let $f\in \mathcal{P}$ and ${\bf u}\in \mathcal{P}^*$. Then
\begin{align}
\mathrm{D}_s ^n\mathrm{S}_sf=\mathrm{S}_s\mathrm{D}_s ^n f,\quad {\bf D}_s ^n{\bf S}_s{\bf u}={\bf S}_s{\bf D}_s ^n {\bf u}. \label{DxnSxf-DxnSxu}
\end{align}
\end{proposition}
\begin{proposition}
Let $z=x(s)=s+\gamma(1+(-1)^s)$, then the following holds:
\begin{align}
\mathrm{D}_s z^n &=n z^{n-1}-2n(n-1)(-1)^s\gamma z^{n-2}+u_nz^{n-3}+v_nz^{n-4}+w_nz^{n-5} +\cdots, \label{Dx-xn}\\[7pt]
\mathrm{S}_s z^n &= z^n-2n(-1)^s\gamma z^{n-1} +\widehat{u}_n z^{n-2}+\widehat{v}_nz^{n-3}+\widehat{w}_n z^{n-4}+\cdots, \label{Sx-xn}
\end{align}
 where
\begin{align}
u_n &:=\frac{n(n-1)(n-2)}{6}(1+12\gamma ^2), \label{unD} \\[7pt]
v_n &:=-(-1)^s\gamma \frac{n(n-1)(n-2)(n-3)}{3}(1+4\gamma ^2),\label{vnD}  \\[7pt]
w_n&:=\frac{n(n-1)(n-2)(n-3)(n-4)}{120}(1+40\gamma ^2 +80\gamma ^4),\label{wnD}\\[7pt]
\widehat{u}_n &:= \frac{n(n-1)}{2}(1+4\gamma ^2), \label{unS} \\[7pt]
\widehat{v}_n &:=-(-1)^s\gamma\frac{n(n-1)(n-2)}{3}(4\gamma ^2 +3), \label{vnS}\\[7pt]
\widehat{w}_n&:=\frac{n(n-1)(n-2)(n-3)}{24}(1+24\gamma ^2 +16\gamma^4). \label{wnS}
\end{align}
\end{proposition}

\begin{proof}
The proof is given by mathematical induction on $n$. 
For $n=0$, we have $\mathrm{D}_s z^0 =\mathrm{D}_s 1=0$ and $\mathrm{S}_s z^0 =1$. 
Since $u_0=v_0=w_0=0$ and $\widehat{u}_0=\widehat{v}_0=\widehat{w}_0=0$, we see that (\ref{Dx-xn})-(\ref{Sx-xn}) hold for $n=0$. 
Next suppose that relations (\ref{Dx-xn})-(\ref{Sx-xn}) are true for all integer numbers less than or equal to a fixed nonnegative integer number $n$. 
Using this hypothesis together with \eqref{def-Dxfg}, we obtain
\begin{align*}
\mathrm{D}_s z^{n+1} &=\mathrm{D}_s z^n ~ \mathrm{S}_s z + \mathrm{S}_s z^n ~\mathrm{D}_s z 
=  \big(z-2(-1)^s\gamma \big)\mathrm{D}_s z^n + \mathrm{S}_s z^n  \\[7pt]
&=z^n-2n(-1)^s\gamma z^{n-1} +\widehat{u}_n z^{n-2}+\widehat{v}_nz^{n-3}+\widehat{w}_n z^{n-4}+\cdots  \\[7pt]
&+\big(z-2(-1)^s\gamma \big)\big( n z^{n-1}-2n(n-1)(-1)^s\gamma z^{n-2}+u_nz^{n-3}+v_nz^{n-4} +\cdots \big)\\[7pt]
&=(n+1)z^{n} -2n(n+1)(-1)^s\gamma z^{n-1} +(u_n+\widehat{u}_n +4n(n-1)\gamma ^2 )z^{n-2}\\[7pt]
&+\big(v_n+\widehat{v}_n -2(-1)^s\gamma u_n\big)z^{n-3}+\big(w_n+\widehat{w}_n-2(-1)^s\gamma v_n \big)z^{n-4}+\cdots.
\end{align*}
The reader should check itself that $u_n+\widehat{u}_n +4n(n-1)\gamma ^2=u_{n+1}$, $v_n+\widehat{v}_n -2(-1)^s\gamma u_n=v_{n+1}$ and $w_n+\widehat{w}_n-2(-1)^s\gamma v_n=w_{n+1}$. So,
\begin{align*}
&\mathrm{D}_s z^{n+1}\\[7pt]
&=(n+1)z^{n} -2n(n+1)(-1)^s\gamma z^{n-1} +u_{n+1}z^{n-2}
+v_{n+1}z^{n-3}+w_{n+1} z^{n-4}+\cdots.
\end{align*}
Similarly,
\begin{align*}
\mathrm{S}_s z^{n+1} &=\mathrm{D}_s z~\mathrm{D}_s z^n  + \mathrm{S}_s z^n ~\mathrm{S}_s z 
= \mathrm{D}_s z^n + \big(z-2(-1)^s\gamma \big)\mathrm{S}_s z^n  \\[7pt]
&=n z^{n-1}-2n(n-1)(-1)^s\gamma z^{n-2}+u_nz^{n-3}+v_nz^{n-4}+w_nz^{n-5} +\cdots\\[7pt]
&+\big(z-2(-1)^s\gamma \big)\big(z^n-2n(-1)^s\gamma z^{n-1} +\widehat{u}_n z^{n-2}+\widehat{v}_nz^{n-3}+\widehat{w}_n z^{n-4}+\cdots  \big)\\[7pt]
&= z^{n+1} -2(n+1)(-1)^s\gamma z^n+(\widehat{u}_n+n+4n\gamma ^2)z^{n-1}\\[7pt]
&+\big(\widehat{v}_n -2n(n-1)(-1)^s\gamma -2(-1)^s\gamma \widehat{u}_n \big)z^{n-2} +(u_n+\widehat{w}_n-2(-1)^s\gamma \widehat{v}_n)z^{n-3} \\[7pt]&+\cdots.
\end{align*}
Again the reader can check that the following identities hold $\widehat{u}_n+n+4n\gamma ^2=\widehat{u}_{n+1}$, $\widehat{v}_n -2n(n-1)(-1)^s\gamma -2(-1)^s\gamma \widehat{u}_n=\widehat{v}_{n+1}$ and
$u_n+\widehat{w}_n-2(-1)^s\gamma \widehat{v}_n=\widehat{w}_{n+1}$. Therefore
\begin{align*}
\mathrm{S}_s z^{n+1}=z^{n+1} -2(n+1)(-1)^s\gamma z^{n} +\widehat{u}_{n+1}z^{n-1}
+\widehat{v}_{n+1}z^{n-2}+\widehat{w}_{n+1} z^{n-3}+\cdots.
\end{align*}
Hence, we obtain (\ref{Dx-xn})-(\ref{Sx-xn}) with $n$ replaced by $n+1$, and the result follows.
\end{proof}

Here we state a useful version of the Leibniz formula, for the left multiplication of a functional by a polynomial. 

\begin{proposition}[Leibniz's formula]\label{Leibniz-rule-NUL}
Let ${\bf u}\in\mathcal{P}^*$ and $f\in\mathcal{P}$. Then
\begin{align}
\mathbf{D}_s ^n \big(f{\bf u}\big)
=\sum_{k=0}^{n} \mathrm{T}_{n,k}f\, \mathbf{D}_s^{n-k} \mathbf{S}_s^k {\bf u}
 \quad (n=0,1,\ldots),\label{leibnizfor-NUL}
\end{align}
where $\mathrm{T}_{n,k}f$ is a polynomial defined as follows:
\begin{align}\label{T00}
\mathrm{T}_{0,0}f &:=f\;;
\end{align}
and for $n\geq 1$,
\begin{align}\label{Tnk}
\mathrm{T}_{n,k}f&:= \mathrm{S}_s \mathrm{T}_{n-1,k}f+\mathrm{D}_s \mathrm{T}_{n-1,k-1}f,
\end{align}
with the conventions $\mathrm{T}_{n,k}f:=0$ whenever $k>n$ or $k<0$. 
\end{proposition}

\begin{corollary}\label{LeibnizCor1}
Let $z=x(s)=s+\gamma(1+(-1)^s)$. Let ${\bf u}\in\mathcal{P}^*$ and $f\in\mathcal{P}_2$. Set $f(z)=az^2+bz+c$ with $a,b,c\in\mathbb{C}$. Then
\begin{align}\label{leibnizfor-degree-pi-2}
\mathbf{D}_s ^n (f{\bf u})
&=\left(f(z)-(-1)^s\gamma (1-(-1)^n)f'(z)+\frac{1}{2}\Big(n+2\gamma ^2 \big(1-(-1)^n\big) \Big)f''(z)   \right) \mathbf{D}_s^n{\bf u}  \\[7pt]
&\quad+n \Big(f'(z)-(-1)^s\gamma \big(1-(-1)^n\big) f''(z) \Big) \mathbf{D}_s^{n-1}\mathbf{S}_s{\bf u}\nonumber \\[7pt]
&\quad+\frac{n(n-1)}{2}f''(z)  \,\mathbf{D}_s^{n-2}\mathbf{S}_s^2{\bf u}.  \nonumber
\end{align}
In particular, 
\begin{align} \label{leib-for-degree-pi-1}
\mathbf{D}_s ^n \big((az+b){\bf u}\big)= \Big(a\big(z-(-1)^s\gamma (1-(-1)^n)\big)+b\Big)\mathbf{D}_s^n{\bf u} +an \mathbf{D}_s^{n-1} \mathbf{S}_s {\bf u}.
\end{align}
\end{corollary}

\begin{proof} 
It is enough to prove \eqref{leibnizfor-degree-pi-2}. Since $f(z)=az^2+bz+c$, then taking into account \eqref{leibnizfor-NUL}, we only need to show that 
\begin{align}
&(\mathrm{T}_{n,0}f)(z)=f(z)-(-1)^s\gamma (1-(-1)^n)f'(z)+\frac{1}{2}\Big(n+2\gamma ^2 \big(1-(-1)^n\big) \Big)f''(z) 
 ,\label{Tn0-a} \\[7pt]
&(\mathrm{T}_{n,1}f)(z)=n \Big(f'(z)-(-1)^s\gamma \big(1-(-1)^n\big) f''(z) \Big), \label{Tn1-a}\\[7pt]
&(\mathrm{T}_{n,2}f)(z)=n(n-1)a, \label{Tn2-a}
\end{align}
where $\mathrm{T}_{n,k}f$ is defined by \eqref{T00}-\eqref{Tnk}.
Note that 
\begin{align*}
(\mathrm{T}_{n,0}f)(z)&= a\big(z-(-1)^s\gamma (1-(-1)^n)\big)^2 +b\big(z-(-1)^s\gamma (1-(-1)^n)\big) +an+c,\\[7pt]
 (\mathrm{T}_{n,1}f)(z)&=2an\big(z-(-1)^s\gamma (1-(-1)^n)\big) +bn. 
\end{align*}
Equation \eqref{leibnizfor-degree-pi-2} is trivial for $n=0$ since $(\mathrm{T}_{0,0}f)(z)=f(z)$ and $(\mathrm{T}_{0,-1}f)(z)=0=(\mathrm{T}_{0,-2}f)(z)$. Now suppose that \eqref{Tn0-a}-\eqref{Tn2-a} are true for some $k=0,1,\ldots,n$. Since $\mathrm{D}_sz^2=2z-4(-1)^s\gamma $, $\mathrm{S}_sz^2=z^2 -4(-1)^s\gamma z +4\gamma ^2+1$, $\mathrm{S}_s z=z-2(-1)^s\gamma$, then using \eqref{Tnk} and the induction hypothesis, we obtain
\begin{align*}
(\mathrm{T}_{n+1,0}&f)(z)= (\mathrm{S}_s\mathrm{T}_{n,0}f)(z)\\[7pt]
&=a\big(z^2 -4(-1)^s\gamma z +4\gamma ^2+1 \big) +b\big(z-2(-1)^s\gamma  \big) \\[7pt]
&\quad +2a\gamma (1-(-1)^n)((-1)^sz-2\gamma) +an +c +\gamma (1-(-1)^n)(2a\gamma +(-1)^sb)\\[7pt]
&=az^2 +\big(b-2a\gamma (-1)^s(1+(-1)^n)  \big)z +a(n+1)+c \\[7pt]
&\quad +\gamma (1+(-1)^n)(2a\gamma -(-1)^sb).
\end{align*}
Therefore \eqref{Tn0-a} is true for $n$ replaced by $n+1$. Similarly, from \eqref{Tnk} and the induction hypothesis, we obtain
\begin{align*}
(\mathrm{T}_{n+1,1}f)(z)&=(\mathrm{S}_s\mathrm{T}_{n,1}f)(z) +(\mathrm{D}_s\mathrm{T}_{n,0}f)(z)\\[7pt]
&=2an\big(z-2(-1)^s\gamma\big)+bn+2an(-1)^s\gamma (1-(-1)^n) \\[7pt]
&\quad+a\big(2z-4(-1)^s\gamma\big)+b+2a(-1)^s\gamma (1-(-1)^n)\\[7pt]
&=2a(n+1)z +b(n+1) -2a(-1)^s\gamma (n+1)(1+(-1)^n).
\end{align*} 
This also shows that \eqref{Tn1-a} is true for $n$ replaced by $n+1$. Finally, as in the previous cases, we obtain
\begin{align*}
(\mathrm{T}_{n+1,2}f)(z)&=(\mathrm{S}_s\mathrm{T}_{n,2}f)(z) +(\mathrm{D}_s\mathrm{T}_{n,1}f)(z)\\[7pt]
&=n(n-1)a+2an =an(n+1),
\end{align*}
and so \eqref{Tn2-a} also becomes true for $n$ replaced by $n+1$, and the proof is completed. 
\end{proof}

We denote by $P_n ^{[k]}$ the monic polynomial of degree $n$ defined by
\begin{align}
P_n ^{[k]} (z):=\frac{D_x ^k P_{n+k} (z)}{ \prod_{j=1} ^k (n+j)} =\frac{n !}{(n+k) !} D_x ^k P_{n+k} (z) \quad (k,n=0,1,\ldots). \label{Pnkx}
\end{align}
Here, as usual, it is understood that $\mathrm{D}_s ^0 f=f $, empty product equals one.

\begin{definition}
Let $\phi\in\mathcal{P}_2$ and $\psi\in\mathcal{P}_1$.
$(\phi,\psi)$ is called an admissible pair if
$$
d_n\equiv d_n(\phi,\psi,x):=\frac12\,n\,\phi^{\prime\prime}+\psi'\neq0 \quad (n=0,1,\ldots).
$$
\end{definition}

This is an analogous for bi-lattices of the original definition of Maroni (cf. \cite{M1991}).

\begin{lemma} \label{Charact-theo}\quad
Let ${\bf u}$ be a linear functional and $\phi$, $\psi$ two polynomials of degree at most two and one, respectively, such that
\begin{align}\label{pearson-eq}
{\bf D}_s(\phi {\bf u})={\bf S}_s(\psi {\bf u}).
\end{align}
Then for any nonzero integer $k$,
\begin{align}
&{\bf D}_s(\phi^{[k]} {\bf u}^{[k]})={\bf S}_s(\psi^{[k]} {\bf u}^{[k]}),\label{funct-eq-uk}
\end{align}
where
\begin{align}
& \phi^{[0]}:=\phi,\quad \psi^{[0]}:=\psi, \label{def-phi0psi0}\\[7pt]
& \phi^{[k+1]}:=\mathrm{S}_s\phi^{[k]}+\mathrm{D}_s\psi^{[k]}, \label{def-phik}\\[7pt]
& \psi^{[k+1]}:=\mathrm{D}_s\phi^{[k]}+\mathrm{S}_s\psi^{[k]}, \label{def-psik}
\end{align}
and functionals ${\bf u}^{[k]}\in\mathcal{P}^*$defined by
\begin{equation}\label{uk-func-Dx}
{\bf u}^{[0]}:={\bf u},\quad
{\bf u}^{[k+1]}:=\mathbf{D}_s\big(\psi^{[k]}{\bf u}^{[k]}\big)-\mathbf{S}_s\big(\phi^{[k]}{\bf u}^{[k]}\big).
\end{equation}
\end{lemma}

\begin{proof}
We first claim that 
\begin{align}\label{psik-inv-Sx2}
\mathrm{S}_s ^2\psi ^{[k]}=\psi ^{[k]}.
\end{align}
Indeed, if $k=0$, this is trivial since $\mathrm{S}_s ^2 z=z$, and so $\mathrm{S}_s ^2\psi =\psi$. Now with $k$ fixed in $\mathbb{N}$, we use \eqref{def-psik}, \eqref{DxnSxf-DxnSxu}, \eqref{def-Sx2f} and the fact that $\mathrm{D}_s ^2\psi ^{[k-1]}=0$, $\mathrm{D}_s ^3 \phi ^{[k-1]}=0$ to obtain
\begin{align*}
\mathrm{S}_s ^2\psi ^{[k]}&=\mathrm{S}_s ^2\big( \mathrm{D}_s\phi^{[k-1]}+\mathrm{S}_s\psi^{[k-1]}\big)  \\[7pt]
&=\mathrm{S}_s\big( \mathrm{D}_s \mathrm{S}_s\phi ^{[k-1]} +\mathrm{S}_s ^2 \psi ^{[k-1]}
\big)=\mathrm{S}_s\big( \mathrm{D}_s \mathrm{S}_s\phi ^{[k-1]} +\mathrm{D}_s ^2 \psi ^{[k-1]}
+\psi ^{[k-1]}\big)\\[7pt]
&=\mathrm{D}_s\big(\mathrm{S}_s ^2\phi ^{[k-1]}  \big)+\mathrm{S}_s \psi ^{[k-1]}\\[7pt]
&=\mathrm{D}_s\big(\mathrm{D}_s ^2\phi ^{[k-1]}+\phi ^{[k-1]}  \big)+\mathrm{S}_s \psi ^{[k-1]}= \mathrm{D}_s\phi ^{[k-1]}+\mathrm{S}_s \psi ^{[k-1]}=\psi ^{[k]},
\end{align*}
and so \eqref{psik-inv-Sx2} holds. Now by induction on $k\in \mathbb{N}_0$. Equation \eqref{funct-eq-uk} is trivial for $k=0$. Suppose that \eqref{funct-eq-uk} is satisfied for some $n=0,1,\ldots, k$. Using \eqref{def-phik}, \eqref{uk-func-Dx}, \eqref{def-Dxfu}, \eqref{DxnSxf-DxnSxu} and the induction hypothesis, we write
\begin{align*}
{\bf D}_s \big(\phi^{[k+1]}{\bf u} ^{[k+1]}  \big)&= {\bf D}_s \Big(\big( \mathrm{S}_s\phi^{[k]}+\mathrm{D}_s\psi^{[k]} \big)\Big( \mathbf{D}_s\big(\psi^{[k]}{\bf u}^{[k]}\big)-\mathbf{S}_s\big(\phi^{[k]}{\bf u}^{[k]}\big) \Big)   \Big)\\
&=\big(\mathrm{D}_s\mathrm{S}_s\phi ^{[k]} +\mathrm{D}_s ^2\psi ^{[k]}  \big)\Big({\bf D}_s{\bf S}_s(\psi ^{[k]} {\bf u} ^{[k]}) -  { {\bf S}_s ^2}  (\phi ^{[k]}{\bf u} ^{[k]})  \Big) \\[7pt]
&\quad+\big(\mathrm{S}_s ^2\phi ^{[k]} +\mathrm{D}_s \mathrm{S}_s\psi ^{[k]}  \big)\Big({\bf D}_s ^2(\psi ^{[k]} {\bf u} ^{[k]}) -{ {\bf D}_s} {\bf S}_s (\phi ^{[k]}{\bf u} ^{[k]})  \Big) \\[7pt]
&=\mathrm{D}_s\mathrm{S}_s\phi ^{[k]} \Big({\bf D}_s ^2(\phi ^{[k]} {\bf u} ^{[k]}) -  {\bf S}_s ^2  (\phi ^{[k]}{\bf u} ^{[k]})  \Big) \\[7pt]
&\quad+\big(\mathrm{D}_s ^2\phi ^{[k]}+\phi ^{[k]} +\mathrm{D}_s \mathrm{S}_s\psi ^{[k]}  \big)\Big({\bf D}_s ^2(\psi ^{[k]} {\bf u} ^{[k]}) -{\bf S}_s ^2(\psi ^{[k]}{\bf u} ^{[k]})  \Big).
\end{align*}
We then use \eqref{def-Dx2u-Sx2u} to deduce 
\begin{align}\label{one-proof-funct-for-uk-phik-psik}
{\bf D}_s \big(\phi^{[k+1]}{\bf u} ^{[k+1]}  \big)=-\Big(\phi ^{[k]}\mathrm{D}_s\mathrm{S}_s \phi ^{[k]} +\psi ^{[k]}\big( \mathrm{D}_s ^2\phi ^{[k]}+\phi ^{[k]} +\mathrm{D}_s \mathrm{S}_s\psi ^{[k]}  \big)  \Big){\bf u}^{[k]}.
\end{align}
Similarly, using \eqref{def-phik}, \eqref{uk-func-Dx}, \eqref{def-Sxfu}, { \eqref{DxnSxf-DxnSxu}}, \eqref{psik-inv-Sx2} and the induction hypothesis, we write
\begin{align*}
{\bf S}_s \big(\psi^{[k+1]}{\bf u} ^{[k+1]}  \big)&= {\bf S}_s \Big(\big( \mathrm{D}_s\phi^{[k]}+\mathrm{S}_s\psi^{[k]} \big)\Big( \mathbf{D}_s\big(\psi^{[k]}{\bf u}^{[k]}\big)-\mathbf{S}_s\big(\phi^{[k]}{\bf u}^{[k]}\big) \Big)   \Big)\\
&=\big(\mathrm{D}_s ^2 \phi ^{[k]} +\mathrm{D}_s \mathrm{S}_s \psi ^{[k]}  \big)\Big({\bf D}_s ^2(\psi ^{[k]} {\bf u} ^{[k]}) - {\bf S}_s ^2  (\psi ^{[k]}{\bf u} ^{[k]})  \Big) \\[7pt]
&\quad+\big(\mathrm{S}_s ^2\psi ^{[k]} +\mathrm{D}_s \mathrm{S}_s\phi ^{[k]}  \big)\Big({\bf D}_s ^2(\phi ^{[k]} {\bf u} ^{[k]}) -{\bf S}_s ^2(\phi ^{[k]}{\bf u} ^{[k]})  \Big) \\[7pt]
&=-\Big(\phi ^{[k]}\mathrm{D}_s\mathrm{S}_s \phi ^{[k]}+\phi ^{[k]}\mathrm{S}_s ^2\psi ^{[k]}  +\psi ^{[k]}\mathrm{D}_s ^2\phi ^{[k]}+\psi ^{[k]}\mathrm{D}_s \mathrm{S}_s\psi ^{[k]}  \big)  \Big){\bf u}^{[k]}\\[7pt]
&=-\Big(\phi ^{[k]}\mathrm{D}_s\mathrm{S}_s \phi ^{[k]}+\phi ^{[k]}\psi ^{[k]}  +\psi ^{[k]}\mathrm{D}_s ^2\phi ^{[k]}+\psi ^{[k]}\mathrm{D}_s \mathrm{S}_s\psi ^{[k]}  \big)  \Big){\bf u}^{[k]}.
\end{align*}
From \eqref{one-proof-funct-for-uk-phik-psik}, we get
${\bf D}_s (\phi^{[k+1]}{\bf u} ^{[k+1]} )={\bf S}_s (\psi^{[k+1]}{\bf u} ^{[k+1]}) $ and hence \eqref{funct-eq-uk} holds of $n=k+1$, and therefore for all $k$.
\end{proof}

\begin{proposition}
Let $\phi\in\mathcal{P}_2$ and $\psi\in\mathcal{P}_1$, so there are $a,b,c,d,e\in\mathbb{C}$ such that
$$\phi(z)=az^2+bz+c,\quad \psi(z)=dz+e.$$ 
Then the polynomials $\phi^{[k]}$ and $\psi^{[k]}$
defined by \eqref{def-phi0psi0}-\eqref{def-psik} are given by 
\begin{align}
\psi^{[k]}(z)&=(2ak+d)\big(z-(-1)^s\gamma (1-(-1)^k)\big)+bk+e, \label{psi-explicit}\\
\phi^{[k]}(z)&=a\big(z-(-1)^s\gamma (1-(-1)^k)\big)^2 +b\big(z-(-1)^s\gamma (1-(-1)^k)\big) +c+k(ak+d).  \label{phi-explicit} 
\end{align}
for each $k=0,1,2\ldots$.
\end{proposition}

\begin{proof}
It is obvious that with expressions \eqref{psi-explicit}-\eqref{phi-explicit}, we obtain
$\phi ^{[0]}(z)=\phi(z)$ and $\psi ^{[0]}(z)=\psi(z)$. Now taking into account the following identities
\begin{align*}
&\mathrm{D}_s \big(z-(-1)^s\gamma(1-(-1)^k)  \big)=1,\\[7pt]
&\mathrm{S}_s \big(z-(-1)^s\gamma(1-(-1)^k)  \big)=z-(-1)^s\gamma (1-(-1)^{k+1}),\\[7pt]
&\mathrm{D}_s \big(z-(-1)^s\gamma(1-(-1)^k)  \big)^2= 2\big(z-(-1)^s\gamma(1-(-1)^{k+1}) \big),\\[7pt]
&\mathrm{S}_s \big(z-(-1)^s\gamma(1-(-1)^k)  \big)^2=\Big(z-(-1)^s\gamma (1-(-1)^{k+1})\Big)^2+1,
\end{align*}
we obtain
\begin{align*}
\mathrm{S}_s\phi ^{[k]}(z)+\mathrm{D}_s\psi ^{[k]}(z)&=a\big(z-(-1)^s\gamma(1-(-1)^{k+1})  \big)^2 +c+a+k(ak+d)\\[7pt]
&\quad+b\big(z-(-1)^s\gamma (1-(-1)^{k+1})  \big) +(2ak+d)\\[7pt]
&=\phi ^{[k+1]}(z).
\end{align*}
Similarly, we also obtain
\begin{align*}
\mathrm{S}_s\psi ^{[k]}(z)+\mathrm{D}_s\phi ^{[k]}(z)&=(2a+2ak+d)\big(z-(-1)^s\gamma(1-(-1)^{k+1})  \big) +bk+e +b\\[7pt]
&=\psi ^{[k+1]}(z).
\end{align*}
Thus \eqref{psi-explicit}-\eqref{phi-explicit} follow.
\end{proof}

\begin{lemma}\label{lemmaA}
Let ${\bf u} \in \mathcal{P}^*$ and suppose that there exist $\phi\in\mathcal{P}_2$ and $\psi\in\mathcal{P}_1$ such that ${\bf u}$ satisfies \eqref{NUL-Pearson}. Then the relations
\begin{align}
\mathbf{D}_s{\bf u}^{[k+1]}
=-\psi^{[k]}{\bf u}^{[k]}, \quad
\mathbf{S}_s{\bf u}^{[k+1]}=-\phi^{[k]}{\bf u}^{[k]}, \label{lemmaA2}
\end{align}
hold for each $k=0,1,\ldots$.
\end{lemma}

\begin{proof}
Let $\phi\in\mathcal{P}_2$ and $\psi\in\mathcal{P}_1$ and ${\bf u}\in \mathcal{P}^*$ such the functional equation \eqref{NUL-Pearson} holds. Using successively \eqref{def-Dx2u-Sx2u}, {\eqref{DxnSxf-DxnSxu}} and \eqref{funct-eq-uk}, we obtain
\begin{align*}
{\bf D}_s {\bf u}^{[k+1]}&={\bf D}_s ^2\big(\psi^{[k]}{\bf u}^{[k]}  \big)-{\bf D}_s{\bf S}_s\big(\phi ^{[k]}{\bf u}^{[k]}  \big)\\[7pt]
&={\bf S}_s ^2\big(\psi^{[k]}{\bf u}^{[k]}  \big)-\psi^{[k]}{\bf u}^{[k]}-{\bf S}_s ^2\big(\psi ^{[k]}{\bf u}^{[k]}  \big)\\[7pt]
&=-\psi^{[k]}{\bf u}^{[k]}.
\end{align*}
In a similar way,
\begin{align*}
{\bf S}_s {\bf u}^{[k+1]}&={\bf S}_s{\bf D}_s\big(\psi^{[k]}{\bf u}^{[k]}  \big)-{\bf S}_s^2\big(\phi ^{[k]}{\bf u}^{[k]}  \big)\\[7pt]
&={\bf D}_s ^2\big(\phi^{[k]}{\bf u}^{[k]}  \big)-\phi^{[k]}{\bf u}^{[k]}-{\bf D}_s ^2\big(\phi ^{[k]}{\bf u}^{[k]}  \big)\\[7pt]
&=-\phi^{[k]}{\bf u}^{[k]},
\end{align*}
and the result follows.
\end{proof}

Here we prove a functional version of the Rodrigues formula on the bi-lattice \eqref{xs-def}.

\begin{theorem}[Rodrigues' formula]\label{uk-is-classical}
Consider the bi-lattice \eqref{xs-def}. Let ${\bf u}\in\mathcal{P}^*$ and suppose that there exists a admissible pair $(\phi,\psi)$ such that
${\bf u}$ fulfills \eqref{NUL-Pearson}.
Set
\begin{equation}\label{dnen-Prop}
d_n:=\mbox{$\frac{1}{2}\,$}\phi'' n+\psi',\quad
e_n:=\phi'(0)n+\psi(0).
\end{equation}
Then 
\begin{align}\label{roformula}
R_n {\bf u} = \mathbf{D}_s^n {\bf u}^{[n]},
\end{align}
where ${\bf u}^{[n]}$ is the functional on $\mathcal{P}$ defined by \eqref{uk-func-Dx} and $(R_n)_{n\geq0}$ is a simple set of polynomials given by
\begin{align}\label{Rn-Prop1} 
R_{n+1}(z)=(a_n z-s_n)R_{n}(z)- t_n R_{n-1} (z), 
\end{align}
with initial conditions $R_{-1}=0$ and $R_0=1$,
and $(a_n)_{n\geq0}$, $(s_n)_{n\geq0}$, and $(t_n)_{n\geq1}$
are sequences of complex numbers defined by
\begin{align}
&a_n:=-\frac{\,d_{2n}d_{2n-1}}{d_{n-1}}, \label{rn-Prop1}\\[7pt]
&s_n:=a_n \left( \frac{n e_{n-1}}{d_{2n-2}}
-\frac{(n+1) e_{n}}{d_{2n}} \right) , \label{sn-Prop1}\\[7pt]
&t_n:=a_n\,\frac{n d_{2n-2}}{d_{2n-1}}\phi^{[n-1]}\left((-1)^s\gamma (1+(-1)^n) -\frac{e_{n-1}}{d_{2n-2}}\right),\label{tn-Prop1}
\end{align}
$\phi^{[n-1]}$ being given by \eqref{phi-explicit}.
(It is understood that $a_0:=- d$ and $s_0:= e$.)
\end{theorem}

\begin{proof}
We apply mathematical induction on $n$.
If $n=0$, (\ref{roformula}) is trivial.
If $n=1$, (\ref{roformula}) follows from the first equation in \eqref{lemmaA2} with $R_1=-\psi$.
Assume now (induction hypothesis) that (\ref{roformula}) holds for two consecutive nonnegative integer numbers,
i.e., the relations
\begin{align}\label{InducHyp}
R_{n-1} {\bf u}=\mathbf{D}_s^{n-1} {\bf u}^{[n-1]} ,\quad
R_n {\bf u}=\mathbf{D}_s^n {\bf u}^{[n]}
\end{align}
hold for some fixed $n\in\mathbb{N}$.
We need to prove that $R_{n+1}{\bf u}=\mathbf{D}_s^{n+1} {\bf u}^{[n+1]}$.
Notice first that, by \eqref{psi-explicit} and \eqref{dnen-Prop}, we have
\begin{equation}\label{psid2n}
\psi^{[n]}(z)=d_{2n}\big(z-(-1)^s\gamma (1-(-1)^n) \big)+e_n.
\end{equation}
By \eqref{lemmaA2} and the Leibniz formula in Proposition \ref{Leibniz-rule-NUL}, we may write
\begin{align*}
\mathbf{D}_s^{n+1}{\bf u}^{[n+1]}
&= \mathbf{D}_s^n \mathbf{D}_s {\bf u}^{[n+1]} = - \mathbf{D}_s^n (\psi^{[n]} {\bf u}^{[n]}) \\[7pt]
&= - \mathrm{T}_{n,0} \psi ^{[n]}  \mathbf{D}_s^n {\bf u}^{[n]}
- \mathrm{T}_{n,1}\psi^{[n]}\mathbf{D}_s^{n-1} \mathbf{S}_s {\bf u}^{[n]}.
\end{align*}
From \eqref{leib-for-degree-pi-1} we have $\mathrm{T}_{n,1}\psi^{[n]}=nd_{2n}$,
and so, using also \eqref{InducHyp},
\begin{align}
\mathbf{D}_s^{n-1} \mathbf{S}_s {\bf u}^{[n]} =-\frac{1}{n d_{2n}}
\left(\mathbf{D}_s ^{n+1} {\bf u}^{[n+1]}
+  \big(\mathrm{T}_{n,0}\psi ^{[n]}\big) R_n {\bf u} \right). \label{1a-lattice}
\end{align}
Shifting $n$ into $n-1$, and using again the induction hypothesis \eqref{InducHyp}, we obtain
\begin{align}
\mathbf{D}_s^{n-2} \mathbf{S}_s {\bf u}^{[n-1]} =-\frac{1}{(n-1)d_{2n-2}}
\left(R_n+ \big(\mathrm{T}_{n-1,0}\psi^{[n-1]}\big)\,R_{n-1}\right){\bf u}. \label{1b-lattice}
\end{align}
Next, using \eqref{def-Dxfu} and \eqref{lemmaA2}, we deduce
\begin{align}
\mathbf{D}_s^{n+1}{\bf u}^{[n+1]} &= - \mathbf{D}_s^n \big(\psi^{[n]}{\bf u}^{[n]}\big)
=- \mathbf{D}_s^{n-1}\big( \mathbf{D}_s(\psi^{[n]} {\bf u}^{[n]})\big) \label{Eq-xi1}\\[7pt]
&=-\mathbf{D}_s^{n-1} \Big(  \mathrm{S}_s \psi ^{[n]}
\mathbf{D}_s{\bf u}^{[n]} +\mathrm{D}_s\psi ^{[n]}\mathbf{S}_s {\bf u}^{[n]} \Big) \nonumber \\[7pt]
&=\mathbf{D}_s^{n-1} \big( \xi_2(\cdot;n) {\bf u}^{[n-1]}  \big), \nonumber
\end{align}
where $\xi_2(\cdot;n)$ is a polynomial of degree $2$, given by
\begin{align}\label{xi-definition}
\xi_2 (z;n)=\psi ^{[n-1]}(z)\mathrm{S}_s \psi ^{[n]}(z) +\phi ^{[n-1]}(z)\mathrm{D}_s \psi ^{[n]}(z).
\end{align}
We deduce
\begin{align}\label{xi-2zn}
\xi_2 (z;n) &= d_{2n}d_{2n-1}\Big(z-(-1)^s\gamma \big(1+(-1)^n \big)  \Big)^2 +2e_nd_{2n-1}\Big(z-(-1)^s\gamma \big(1+(-1)^n \big)  \Big) \\
&\quad  +e_ne_{n-1} +d_{2n}(c+(n-1)d_{n-1}).\nonumber
\end{align}
Since $\deg\xi_2(\cdot;n)=2$, using again Proposition \ref{Leibniz-rule-NUL}, we may write
\begin{align}
\mathbf{D}_s^{n-1}\big(\xi_2(\cdot;n){\bf u}^{[n-1]}\big)
&= \mathrm{T}_{n-1,0}\xi_2(\cdot;n) \mathbf{D}_s^{n-1}{\bf u}^{[n-1]}
+\mathrm{T}_{n-1,1}\xi_2(\cdot;n)\mathbf{D}_s^{n-2}\mathrm{S}_s{\bf u}^{[n-1]} \label{Eqq-ss} \\[7pt]
&\quad+\mathrm{T}_{n-1,2}\xi_2(\cdot;n) \mathbf{D}_s^{n-3}\mathbf{S}_s^2{\bf u}^{[n-1]}. \nonumber
\end{align}
Since, by \eqref{leib-for-degree-pi-1},
$\mathrm{T}_{n-1,2}\xi_2(\cdot;n)= (n-1)(n-2)d_{2n} d_{2n-1}$,
combining equations \eqref{Eqq-ss}, \eqref{Eq-xi1}, \eqref{1b-lattice}, and \eqref{InducHyp}, we obtain
\begin{align}
\mathbf{D}_s^{n-3} \mathbf{S}_s^2{\bf u}^{[n-1]}
&=\frac{1}{(n-1)(n-2)d_{2n}d_{2n-1}}
\left( \mathbf{D}_s^{n+1}{\bf u}^{[n+1]}-\big(\mathrm{T}_{n-1,0}\xi_2(\cdot;n)\big)R_{n-1}{\bf u} \right.\label{1c-lattice} \\[7pt]
&\quad + \left. \frac{ \mathrm{T}_{n-1,1}\xi_2(\cdot;n)}{(n-1) d_{2n-2}} \Big(R_n  +  \big(\mathrm{T}_{n-1,0}\psi^{[n-1]}\big)R_{n-1} \Big){\bf u} \right). \nonumber
\end{align}
On the other hand, by (\ref{lemmaA2}),
\begin{align}\label{eta-definition}
\mathbf{S}_s {\bf u}^{[n]}=\eta_2(\cdot;n){\bf u}^{[n-1]},\quad
\eta_2(z;n):=-\phi^{[n-1]}(z).
\end{align}
Taking into account that $\eta_2(\cdot;n)$ is a polynomial of degree at most two, by the Leibniz formula and \eqref{InducHyp}, we may write
\begin{align}
\mathbf{D}_s^{n-1} \mathbf{S}_s {\bf u}^{[n]}
&=\mathbf{D}_s^{n-1} \big( \eta_2 (\cdot;n){\bf u}^{[n-1]} \big) \label{EqBoa}\\[7pt]
&= \big(\mathrm{T}_{n-1,0}\eta_2(\cdot;n)\big)R_{n-1}{\bf u}
+\mathrm{T}_{n-1,1}\eta_2(\cdot;n)\mathbf{D}_s^{n-2} \mathbf{S}_s{\bf u}^{[n-1]} \nonumber \\[7pt]
&\quad +\mathrm{T}_{n-1,2}\eta_2(\cdot;n)\mathbf{D}_s^{n-3} \mathbf{S}_s^2{\bf u}^{[n-1]}. \nonumber
\end{align}
Note that $\eta_2(\cdot;n)$ is given explicitly by
\begin{align}\label{eta-2zn}
\eta_2 (z;n)&=-a\Big(z-(-1)^s\gamma \big(1+(-1)^n \big)  \Big)^2 -b\Big(z-(-1)^s\gamma \big(1+(-1)^n \big)  \Big) -c-(n-1)d_{n-1}
\end{align}
Hence, using \eqref{leib-for-degree-pi-1},
$\mathrm{T}_{n-1,2}\eta_2(\cdot;n)=-a(n-1)(n-2)$.
Therefore, substituting \eqref{1a-lattice}, \eqref{1b-lattice}, and \eqref{1c-lattice} in \eqref{EqBoa}, we obtain
\begin{equation}\label{Dn1HypInd}
\mathbf{D}_s^{n+1} {\bf u}^{[n+1]}
=\big(A(\cdot;n) R_n+ B(\cdot;n)R_{n-1}\big){\bf u},
\end{equation}
where $A(\cdot;n)$ and $B(\cdot;n)$ are polynomials depending on $n$, given by
\begin{align}\label{Azn-initial}
\epsilon_n A(z;n) = \frac{ \big(\mathrm{T}_{n,0}\psi ^{[n]}\big)(z)}{n d_{2n}} -\frac{ \big(\mathrm{T}_{n-1,1} \eta_2 \big)(z;n)}{(n-1) d_{2n-2}} -\frac{ a\big(\mathrm{T}_{n-1,1}\xi_2\big)(z;n)}{(n-1) d_{2n}d_{2n-1} d_{2n-2}} 
\end{align}
and
\begin{align}\label{Bzn-initial}
\epsilon_n B(z;n) &= \big(\mathrm{T}_{n-1,0}\eta_2\big)(z;n) -\frac{ \big(\mathrm{T}_{n-1,0}\psi ^{[n-1]}\big)(z)\big(\mathrm{T}_{n-1,1}\eta_2\big)(z;n)}{ (n-1) d_{2n-2}} \\[7pt]
&\quad + a\frac{ \big(\mathrm{T}_{n-1,0}\xi_2\big)(z;n)}{ d_{2n} d_{2n-1}} 
 -a \frac{ \big(\mathrm{T}_{n-1,1}\xi_2\big)(z;n)\big(\mathrm{T}_{n-1,0}\psi ^{[n-1]}\big)(z)}{(n-1) d_{2n} d_{2n-1} d_{2n-2}},\nonumber
\end{align}
where
\begin{align*}
\epsilon_n :=\frac{a}{ d_{2n}d_{2n-1}} -\frac{1}{n d_{2n}}=-\frac{d_{n-1}}{n d_{2n} d_{2n-1}}.
\end{align*}
We claim that
\begin{align}
A(z;n)&=- \frac{d_{2n}d_{2n-1}}{d_{n-1}}z +\frac{n d_{2n}d_{2n-1}e_{n-1}}{d_{2n-2}d_{n-1}} 
-\frac{(n+1) d_{2n-1}e_{n}}{d_{n-1}}, \label{Azn-final} \\
&=a_n z-s_n,\nonumber \\
B(z;n)&=  \frac{n d_{2n}d_{2n-2}}{d_{n-1}} \phi ^{[n-1]} \left((-1)^s\gamma \big(1+(-1)^n\big)-\frac{e_{n-1}}{d_{2n-2}} \right) =-t_n, \label{Bzn-final}
\end{align}
where $a_n$, $s_n$, and $t_n$ are given by \eqref{rn-Prop1}-\eqref{tn-Prop1}. 
Indeed, 
\begin{align}
&a= d_{2n-1}- d_{2n-2}, \label{an-app}\\[7pt]
&b=e_n-e_{n-1} ,\label{bn-app}\\[7pt]
&d_{2n}-2 d_{2n-1}+d_{2n-2}=0.\label{dn-app}
\end{align}
By \eqref{xi-2zn}, \eqref{eta-2zn}, and \eqref{psid2n}, and applying \eqref{leibnizfor-degree-pi-2}-\eqref{leib-for-degree-pi-1} together with \eqref{bn-app}, we obtain
\begin{align}
(\mathrm{T}_{n,0}\psi ^{[n]})(z) &=d_{2n}z +e_n \label{Tn-psi-n,0}   ,\\
(\mathrm{T}_{n-1,1} \xi_2 )(z;n)&=2(n-1)d_{2n-1}(d_{2n} z+e_n)-(-1)^s\gamma d_{2n}d_{2n-1}\big(2n-1 -(-1)^n\big), \label{Tn-xi-n-1,1} \\
(\mathrm{T}_{n-1,1}\eta_2 )(z;n)&=-(n-1)(2az+b) +a(-1)^s\gamma \big(2n-1 -(-1)^n\big).   \label{Tn-eta-n-1,1}  
\end{align}
Similarly,
\begin{align}
(\mathrm{T}_{n-1,0}\eta_2)(z;n)&= -az^2-bz-c-(n-1)d_n   \nonumber \\[7pt]
&=\eta_2\big(z+(-1)^s\gamma (1+(-1)^n);n  \big) -a(n-1) \label{Tn-eta-n-1,0} \\[7pt]
(\mathrm{T}_{n-1,0}\xi_2)(z;n)&=d_{2n}d_{2n-1}z^2 +2e_nd_{2n-1}z  \nonumber\\[7pt]
&\quad +e_ne_{n-1}+d_{2n}\big(  c+(n-1)(d_{n-1}+d_{2n-1})\big)    \nonumber \\[7pt]
&= \xi_2 \big(z+(-1)^s\gamma (1+(-1)^n);n\big)+(n-1)d_{2n}d_{2n-1} . \label{Tn-xi-n-1,0}     
\end{align}
Using \eqref{Tn-psi-n,0}, \eqref{Tn-xi-n-1,1}, and \eqref{Tn-eta-n-1,1}, we deduce from \eqref{Azn-initial} that using $b=e_n -e_{n-1}$ that
\begin{align*}
\epsilon_n A(z;n)&=\frac{z}{n}  +\frac{e_{n}}{nd_{2n}}+ \frac{b}{d_{2n-2}} -\frac{2a e_n }{d_{2n}d_{2n-2}},\\[7pt]
&=\frac{1}{n}z -\frac{e_{n-1}}{d_{2n-2}}+ \frac{(n+1) e_n}{n d_{2n}}
\end{align*}
($n=0,1,\ldots$). This gives \eqref{Azn-final}.
From \eqref{Tn-psi-n,0}-\eqref{Tn-xi-n-1,0} it is straightforward to verify that \eqref{Bzn-initial} reduces to 
\begin{align*}
\epsilon_n B(z;n)&=-(c+(n-1)d_{n-1})\big(1-\frac{a}{d_{2n-1}}  \big) +b\frac{e_{n-1}}{d_{2n-2}} +a\frac{e_ne_{n-1}}{d_{2n}d_{2n-1}}-2a\frac{e_ne_{n-1}}{d_{2n}d_{2n-2}}.
\end{align*} 
Moreover, from the definitions of $\xi(.;n)$ and $\eta_2(.;n)$ given in \eqref{xi-definition} and \eqref{eta-definition}, we use \eqref{an-app}-\eqref{dn-app} to obtain
$$
\epsilon_nB(z;n)=-\frac{d_{2n-2}}{d_{2n-1}}\Big(a\frac{e_{n-1} ^2}{d_{2n-2} ^2}-b\frac{e_{n-1}}{d_{2n-2}}+c+(n-1)d_{n-1}  \Big) .
$$
Therefore, using successively \eqref{an-app} and \eqref{bn-app}, we obtain
\begin{align*}
B(z;n)&= \frac{nd_{2n}d_{2n-2}}{d_{n-1}}\Big(a\frac{e_{n-1} ^2}{d_{2n-2} ^2}-b\frac{e_{n-1}}{d_{2n-2}}+c+(n-1)d_{n-1}    \Big)\\[7pt]
&=\frac{n d_{2n}d_{2n-2}}{d_{n-1}}\phi^{[n-1]}\left((-1)^s\gamma \big(1+(-1)^n\big)-\frac{e_{n-1}}{d_{2n-2}} \right),
\end{align*}
and \eqref{Bzn-final} is proved.
It follows from \eqref{Azn-final} and \eqref{Bzn-final} that  \eqref{Dn1HypInd} reduces to $\mathbf{D}_s^{n+1} {\bf u}^{[n+1]}=R_{n+1}{\bf u}$,
which completes the proof.
\end{proof}

\begin{lemma}\label{phi-psi-not-zero}
Let ${\bf u}\in\mathcal{P}^*$ be regular. Suppose that there is 
$(\phi,\psi)\in\mathcal{P}_2\times\mathcal{P}_1 \setminus \{(0,0)\}$ so that \eqref{NUL-Pearson} holds. 
Then neither $\phi$ nor $\psi$ is the zero polynomial, and $\deg\psi=1$.
\end{lemma}

\begin{lemma}\label{x-admissible}
Let ${\bf u}\in\mathcal{P}^*$. Suppose that ${\bf u}$ is regular and satisfies $(\ref{NUL-Pearson})$, 
where $(\phi,\psi)\in\mathcal{P}_2\times\mathcal{P}_1 \setminus \{(0,0)\}$.
Then $(\phi,\psi)$ is a admissible pair and ${\bf u}^{[k]}$ is regular for each $k=1,2,\ldots$.
Moreover, if $(P_n)_{n\geq0}$ is the monic OPS with respect to ${\bf u}$, then
$\big(P_n^{[k]}\big)_{n\geq0}$ is the monic OPS with respect to ${\bf u}^{[k]}$.
\end{lemma}

\section{Proof of Theorem \ref{main-Thm1}}\label{S-main}\label{th1}
Here we reproduce the ideas developed in \cite{KCDMJP2022}. Suppose that ${\bf u}$ is regular. By Lemma \ref{x-admissible}, $(\phi,\psi)$ is a admissible pair, meaning that the first condition in \eqref{le1a} holds, and
$\big(P_j^{[n]}\big)_{j\geq0}$ is the monic OPS with respect to ${\bf u}^{[n]}$, for each fixed $n$.
Write 
\begin{equation}\label{Dx-ttrrC1}
P_{j+1}^{[n]}(z)=(z-B_j^{[n]})P_j^{[n]}(z)-C_j^{[n]}P_{j-1}^{[n]}(z)\quad  (j=0,1,\ldots)
\end{equation}
($P_{-1}^{[n]}(z):=0$), being $B_{j}^{[n]}\in\mathbb{C}$
and $C_{j+1}^{[n]}\in\mathbb{C}\setminus\{0\}$ for each $j=0,1,\ldots$.
Let us compute $C_{1}^{[n]}$. 
We first show that (for $n=0$) the coefficient $C_1\equiv C_{1}^{[0]}$,
appearing in the recurrence relation for $(P_j)_{j\geq0}$,
is given by
\begin{equation}\label{Dx-g1}
C_1=-\frac{1}{d +a}\,\phi\left(-\frac{e}{d}\right) =-\frac{1}{d_1}\,\phi\left( -\frac{e_0}{d_0}\right)\; .
\end{equation}
Indeed, taking $k=0$ and $k=1$ in the relation
$\langle{\bf D}_{x}(\phi{\bf u}),z^k\rangle=\langle {\bf S}_s(\psi{\bf u}),z^k\rangle$,
we obtain $0=du_1+eu_0$ and
$(a+d)u_2 +(b+e-2(-1)^s\gamma d)u_1+(c-2(-1)^s\gamma e)u_0=0$, where
$u_k:=\langle{\bf u},z^k\rangle$ ($k=0,1,\ldots$). Therefore,
\begin{equation}
u_1=-\frac{e}{d}u_0 \; , \quad u_2=-\frac{1}{d+a}\left[-(b+e)\frac{e}{d}+c \right] u_0 \; .\label{Dx-u1}
\end{equation}
On the other hand, since $P_1(z)=z-B_0^{[0]}=z-u_1/u_0$, we also have
\begin{equation}
C_1=\frac{\langle{\bf u},P_1^2\rangle}{u_0}=\frac{u_2u_0-u_1^2}{u_0^2}
=\frac{u_2}{u_0}-\left(\frac{u_1}{u_0}\right)^2 \; .
\label{Dx-u2}
\end{equation}
Substituting $u_1$ and $u_2$ given by (\ref{Dx-u1}) into (\ref{Dx-u2}) yields (\ref{Dx-g1}).
Since equation (\ref{funct-eq-uk}) fulfilled by ${\bf u}^{[n]}$ is of the same type as (\ref{NUL-PearsonMainThm1}) fulfilled by ${\bf u}$, 
with the polynomials $\phi^{[n]}$ and $\psi^{[n]}$ in \eqref{def-phik}-\eqref{def-psik} playing the roles of $\phi$ and $\psi$ in (\ref{NUL-PearsonMainThm1}),  
 $C_{1}^{[n]}$ may be obtained replacing in (\ref{Dx-g1}) $\phi$ and $\psi(z)=dz+e$ by $\phi ^{[n]}$ and $\psi^{[n]}(z)=d_{2n}\big(z-(-1)^s\gamma (1-(-1)^n)\big)+e_n$, respectively. 
 Hence,
\begin{align*}
C_{1}^{[n]}&=-\frac{1}{d_{2n} +a}\phi ^{[n]} \left((-1)^{s}\gamma (1-(-1)^n) -\frac{e_n}{d_{2n}}\right) \\[7pt]
&=-\frac{1}{d_{2n+1}}\phi ^{[n]}\left((-1)^s\gamma (1-(-1)^n) -\frac{e_n}{d_{2n}}\right) .
\end{align*}
Since ${\bf u}^{[n]}$ is regular, then $C_{1}^{[n]}\neq0$, and so 
the second condition in \eqref{le1a} holds.

Conversely, suppose that conditions \eqref{le1a} hold.
Define a sequence of polynomials, $(P_n)_{n\geq 0}$, by \eqref{ttrr-Dx}-\eqref{Cn-Dx}, with $P_{-1}(z):=0$. 
According to the hypothesis \eqref{le1a}, $C_{n+1} \neq 0$ for each $n=0,1,\ldots$. Therefore, Favard's theorem ensures that $(P_n)_{n\geq 0}$ is a monic OPS. 
To prove that ${\bf u}$ is regular we will show that $(P_n)_{n\geq 0}$ is the monic OPS with respect to $\textbf{u}$. 
For that we only need to prove that $$ u_0 \neq 0,\quad\left\langle \textbf{u},P_n \right\rangle =0.$$ 
We start by showing that  $ u_0 \neq 0$. 
Indeed, suppose that  $u_0 =0$. 
Since (\ref{NUL-PearsonMainThm1}) holds, then $\left\langle \mathbf{D}_s (\phi \textbf{u})-\mathbf{S}_s (\psi \textbf{u}), z^n \right\rangle =0$. 
This implies 
\begin{align}
d_n u_{n+1}& +(e_n-2n\gamma (-1)^sd_{n-1})u_n \label{momentseq}\\
&+\big(nc-2n\gamma (-1)^se_{n-1}+n(n-1)(1+4\gamma ^2)d/2   \big)u_{n-1}+\sum_{j=0}^{n-2} a_{n,j} u_j =0 , \nonumber 
\end{align}
$(n=0,1,\ldots)$, for some complex numbers $a_{n,j}$ ($j=0,1,\ldots,n-2$).
Since $u_0=0$ and $d_n\neq0$ for all $n=0,1,\ldots$, from (\ref{momentseq})  we deduce $u_n =0$ for each $n=0,1,\ldots$.
This implies $\textbf{u} = {\bf 0}$, contrary to the hypothesis. Therefore $u_0 \neq 0$.
Next, notice that $P_n (z) =k_n  R_n (z)$ ($n=0,1,\ldots$), where $R_n$ is defined by \eqref{Rn-Prop1} and
$k_n ^{-1}:=(-1)^n \prod_{j=1} ^{n} d_{n+j-2}$. 
Therefore, by the Rodrigues formula \eqref{roformula} given in Theorem \ref{uk-is-classical}, we obtain
\begin{align*}
\langle \textbf{u}, P_n \rangle &= k_n  \langle \textbf{u}, R_n \rangle =k_n  \langle R_n \textbf{u}, 1 \rangle 
= k_n  \big\langle \mathrm{D}_s ^n \textbf{u} ^{[n]}, 1 \big\rangle= (-1)^n k_n  \big\langle \textbf{u}^{[n]}, \mathrm{D}_s ^n ~1 \big\rangle =0 .
\end{align*}
Hence ${\bf u}$ is regular.
The remaining statements in the theorem follow from Theorem \ref{uk-is-classical}.

\section{Proof of Theorem \ref{main-Thm2}}\label{th2}

Let ${\bf u}$ be a regular functional, $\phi$ and $\psi$ two polynomials of degree at most two and one such that ${\bf D}_s(\phi{\bf u})={\bf S}_s(\psi {\bf u})$. Set $\phi (z)=az^2+bz+c$ and $\psi (z)=dz+e$. Since $d\neq 0$, we can always assume that $d=1$.
We distinguish three cases according to the degree of $\phi$.
\\
\\
{\bf Case 1:} If $\deg \phi=0$, that is $\phi(z)=c$ with $c\neq 0$, then from \eqref{Bn-Dx}-\eqref{Cn-Dx}, we obtain
\begin{align*}
B_n= -e  ,\;\quad C_{n+1}=-(n+c)(n+1), 
\end{align*}
with the condition $-c\notin \mathbb{N}_0$. If $(P_n)_{n\geq 0}$ is the monic OPS whose coefficients are the above expressions of $B_n$ and $C_n$, then 
\begin{align*}
P_n(x)=H_n ^{\gamma}\big(z-e;-1,-c\big). 
\end{align*}
\\
{\bf Case 2:} If $\deg \phi=1$, that is $\phi(z)=bz+c$ with $b\neq 0$, then from \eqref{Bn-Dx}-\eqref{Cn-Dx}, we obtain
\begin{align*}
B_n= -2bn-e ,~ C_{n+1}=\big( (b^2-1)n+be-c\big)(n+1), 
\end{align*}
with the condition $(b^2-1)n+be-c \neq 0$, for each $n=0,1,\ldots$. If $(P_n)_{n\geq 0}$ is the monic OPS whose coefficients are the above expressions of $B_n$ and $C_n$, then 
\begin{align*}
P_n(x)=H_n ^{\gamma}\big(z-e;b^2-1,be-c\big). 
\end{align*}
\\
{\bf Case 3:} If $\deg \phi =2$, this means $\phi(z)=az^2+bz+c$, with $a\neq 0$.
Applying \eqref{Bn-Dx}-\eqref{Cn-Dx}, we obtain
\begin{align}
B_n&= -\frac{2abn^2 +2(1-a)bn+(1-2a)e}{(2an-2a+1)(2an+1)},\label{bn-last-proof}\\[7pt]
C_{n+1}&=-\frac{(n+1)(an-a+1)}{(2an-a+1)(2an+1)^2(2an+a+1)}\varphi_4(n),\label{cn-last-proof}
\end{align}
where $$\varphi_4(n)=4a^3n^4 +8a^2n^3 +a(5+4ac-b^2)n^2 +(1-b^2+4ac)n+c-be+ae^2  . $$
Let $\alpha_1$, $\alpha_2$, $\alpha_3$ and $\alpha_4$ be four complex numbers such that
$\varphi_4(n)=4(an+\alpha_1)(an+\alpha_2)(an+\alpha_3)(n+\alpha_4)$. Then we find
\begin{align*}
\alpha_1& =\frac{1}{2}+ \frac{1}{4}\Big((b+1)^2 -4a(e+c)  \Big)^{1/2}+\frac{1}{4}\Big((b-1)^2 -4a(c-e)  \Big)^{1/2},\\[7pt]
\alpha_2 &=\frac{1}{2}+ \frac{1}{4}\Big((b+1)^2 -4a(e+c)  \Big)^{1/2}-\frac{1}{4}\Big((b-1)^2 -4a(c-e)  \Big)^{1/2},\\[7pt]
\alpha_3 &=\frac{1}{2}- \frac{1}{4}\Big((b+1)^2 -4a(e+c)  \Big)^{1/2}+\frac{1}{4}\Big((b-1)^2 -4a(c-e)  \Big)^{1/2},\\[7pt]
\alpha_4 &=\frac{1}{2a}- \frac{1}{4a}\Big((b+1)^2 -4a(e+c)  \Big)^{1/2}-\frac{1}{4a}\Big((b-1)^2 -4a(c-e)  \Big)^{1/2}.
\end{align*}
From the above equations, we rapidly deduce the following identities:
\begin{align*}
&\alpha_4 =(2-\alpha_1-\alpha_2-\alpha_3)/a,~
\alpha_3-a\alpha_4=\alpha_1-\alpha_2,\\[7pt]
&(2\alpha_1-1)(2\alpha_2-1)=b-2ae,\quad (2\alpha_1-1)^2 +(2\alpha_2-1)^2=b^2+1 -4ac.
\end{align*}
We then obtain
\begin{align*}
&\alpha_3=1-\alpha_2,\quad \alpha_4=(1-\alpha_1)/a,\\[7pt]
&e=\big(b-(2\alpha_1-1)(2\alpha_2-1)\big)/(2a)  ,\\[7pt]
&c=\big(b^2+1 -(2\alpha_1-1)^2-(2\alpha_2-1)^2\big)/(4a)   .
\end{align*}
This means that 
\begin{align*}
\phi(z)&=az^2+bz+\big(b^2+1 -r_1 ^2-r_2 ^2\big)/(4a),\\[7pt]
\psi(z)&=x+(b-r_1r_2)/(2a),
\end{align*}
and consequently 
$$\varphi_4(n)=4a^3\Big(n+\frac{1+r_1}{2a}\Big)\Big(n+\frac{1-r_1}{2a}\Big)\Big(n+\frac{1+r_2}{2a}\Big)\Big(n+\frac{1-r_2}{2a}\Big),$$
where
$r_1=2\alpha_1 -1$ and $r_2=2\alpha_2-1$.
Hence \eqref{bn-last-proof}-\eqref{cn-last-proof} become
\begin{align*}
&B_n=  -\frac{b}{2a} -\frac{(2a-1)r_1r_2}{2a(2an-2a+1)(2an+1)},\\[7pt]
&C_{n+1}=\\[7pt]
&-\frac{(n+1)(an-a+1)(2an+r_1+1)(2an-r_1+1)(2an+r_2+1)(2an-r_2+1)}{4a(2an-a+1)(2an+1)^2(2an+a+1)}.
\end{align*}
If $(P_n)_{n\geq 0}$ is the monic OPS whose coefficients are the above expressions of $B_n$ and $C_n$, then 
\begin{align*}
P_n(x)=Q_n ^{\gamma}\Big(z-\frac{b}{2a};\frac{1}{2a},\frac{r_1}{2a},\frac{r_2}{2a}\Big),
\end{align*}
and the proof is completed.

\begin{remark}\quad
Note that the OPS $\big(Q_n ^{\gamma}(\cdot;a,b,c)\big)_{n\geq 0}$ has the following property:
\begin{align*}
Q_n ^{\gamma}(\cdot;a,b,c) =Q_n ^{\gamma}(\cdot;a,c,b)=Q_n ^{\gamma}(\cdot;a,-c,-b).
\end{align*}
Orthogonal polynomial sequences $(H_n ^{\gamma}(\cdot;a,b)_{n\geq 0})$ and $(Q_n ^{\gamma}(\cdot;a,b,c))_{n\geq 0}$ can be obtained from Theorem \ref{main-Thm1} with the pair $(\phi, \psi)=((a+1)^{1/2}z-b,z)$ and $(\phi, \psi  )=(z^2+a^2-b^2-c^2 ,2az-2bc )$, respectively.
\end{remark}

\section*{Acknowledgements }
 This work was supported by the Centre for Mathematics of the University of
Coimbra-UIDB/00324/2020, funded by the Portuguese Government through FCT/
MCTES. KC thanks CMUP, University of Porto, for their support and
hospitality. GF acknowledges the support of the grant entitled ``Geometric approach to ordinary differential equations" funded under New Ideas 3B competition within Priority Research Area III implemented under the ``Excellence Initiative-Research University" (IDUB) Programme (University of Warsaw) (nr 01/IDUB/2019/94). The work of GF is also partially supported by the project PID2021-124472NB-I00 funded by MCIN/AEI/10.13039/501100011033 and by ``ERDF A way of making Europe". This work was carried out during the tenure of an ERCIM Alain Bensoussan Fellowship Programme by the third author and it was concluded during the visit of the first author to the University of Warsaw.

\end{document}